\documentclass[11pt,oneside,reqno]{amsart}
\usepackage{xcolor}
\usepackage{hyperref}
\usepackage{enumitem}

\usepackage{geometry}
\usepackage{amssymb,amsmath,amsthm}
\numberwithin{equation}{section}
\usepackage[english]{babel}
\usepackage{textcomp}
\usepackage{epstopdf}
\usepackage{inputenc}
\usepackage{graphics}

\newcommand\numberthis{\addtocounter{equation}{1}\tag{\theequation}}

\definecolor{armygreen}{rgb}{0.29, 0.33, 0.13}
\definecolor{darkgreen}{rgb}{0.0, 0.2, 0.13}

\newtheorem{thm}{Theorem}[section]
\newtheorem{lem}[thm]{Lemma}

\newtheorem{Def}[thm]{Definition}

\newtheorem{rem}[thm]{Remark}

\newcommand{\Rem}{\begin{rem} \rm}
	\newcommand{\bdfn}{\begin{Def} \rm}
		\newcommand{\edfn}{\end{Def}}

	\newcommand{\ba}{\begin{array}}
		\newcommand{\ea}{\end{array}}

	\newcommand{\compconj}[1]{\overline{#1}}
	
\begin{document}
		\title{Generalized tri-circular projections on some spaces of analytic functions}
\author{Hemant Kumar}
\address{Department of Applied Sciences, Indian Institute of Information Technology Allahabad, Prayagraj-211015, U.P., India.}
\email{hemantkr.math@gmail.com}

\author{Himanshu Kumar}
\address{Department of Applied Sciences, Indian Institute of Information Technology Allahabad, Prayagraj-211015, U.P., India.}
\email{himanshumath1507@gmail.com}

\author{Rahul Maurya}
\address{Department of Mathematics and Statistics, Indian Institute of Technology Kanpur, Kanpur-208016, U.P., India}
\email{rahulmaurya7892@gmail.com}

\author{Abdullah Bin Abu Baker}
\address{Department of Applied Sciences, Indian Institute of Information Technology Allahabad, Prayagraj-211015, U.P., India.}
\email{abdullahmath@gmail.com}

\thanks{The first- and second-named authors gratefully acknowledge the Ministry of Education, New Delhi (India), for financial support, and IIIT Allahabad, India, for providing the resources and infrastructure to carry out this research.  Third-named author is supported by the Department of Mathematics and Statistics, Indian Institute of Technology Kanpur, India. Fourth-named author is partially supported by Anusandhan National Research Foundation (MATRICS) grant No. MTR/2022/000710.}

		
		
		\date{\today}
		
\begin{abstract}
In this paper, we obtain the structure of a class of contractive projections, known as generalized tri-circular projections, on some functional Banach spaces on the open unit disk $\mathbb{D}$, which includes Hardy, Bergman and Bloch spaces. We then consider the space $S^p_{\mathcal{K}}$ of $\mathcal{K}$-valued analytic functions on $\mathbb{D}$ such that $f' \in H^p(\mathcal{K})$ and give complete description of generalized tri-circular projections on this space. Here, $\mathcal{K}$ is a complex separable Hilbert space, and $H^p(\mathcal{K})$ denotes the Hardy space of $\mathcal{K}$-valued analytic functions on $\mathbb{D}$. 			
\end{abstract}
		
\maketitle
		
\thispagestyle{empty}
		
\section{Introduction and preliminary Results}

Projections are basic building blocks in understanding the structure of a Banach space. By a projection on a Banach space $X$ we mean a bounded linear operator $P$ such that $P^2 = P$. We say that a projection $P$ is contractive (respectively, bi-contractive) if $||P|| = 1$ (respectively, $||P|| = ||I - P || = 1$). 

A class of bi-contractive projections, known as generalized bi-circular projections (henceforth $GBP$), have been studied extensively in the last two decades. This class was introduced by Fo\v{s}ner, Ili\v{s}evi\'{c} and Li in 2007 \cite{MDC}. Let $X$ be a Banach space, and let $I$ denotes the identity operator on $X$. A projection $P$ on $X$ is said to be a $GBP$ if $P + \lambda (I - P)$ is a surjective isometry on $X$, where $\mathbb{T}$ denotes the unit circle in the complex plane. In \cite{botelho2009generalized}, Botelho and Jamison characterized $GBPs$ on various Banach spaces of scalar- and vector-valued analytic functions, including Bergman, Bloch and Hardy spaces. Motivated by this work, in this paper, we give complete descriptions of generalized tri-circular projections on these spaces.

\begin{Def} \label{GTP}
Let $X$ be a Banach space. Let $\mathcal{C} = \{P_0,P_1, P_2\}$ be a collection of nonzero distinct projections. Then $\mathcal{C}$ is said to be a family of generalized tri-circular projections ($GTP$, for short) if 
\begin{enumerate}
\item $\mathcal{C}$ defines a partition of the identity by projections: $P_0 + P_1 + P_2 = I$,

\item $\mathcal{C}$ is an orthogonal: $P_iP_j = 0$ for $i \neq j$, $i, j =0,1,2$, and

\item there exist $\lambda_1, \lambda_2 \in \mathbb{T} \setminus \{1\}$ such that $T = P_0 + \lambda_1 P_1 + \lambda_2 P_2$ is a surjective isometry on $X$.
\end{enumerate}
Each $P_i$, $i = 0,1,2$ is called a $GTP$, and we refer to $T$ as the isometry associated with the family $\mathcal{C}$ and with each $P_i$. We also say that $\mathcal{C}$ is a family of $GTPs$ corresponding to an isometry $T$. 
\end{Def}

If $\mathcal{C} = \{P_0,P_1, P_2\}$ is a collection of nonzero distinct projections which defines an orthogonal partition of identity by projections (that is, $P_0 + P_1 + P_2 = I$ and $P_iP_j = 0$ for $i \neq j$, $i, j =0,1,2$) such that $P_0 + \lambda_1 P_1 + \lambda_2 P_2$ is a surjective isometry on $X$ for all $\lambda_1, \lambda_2 \in \mathbb{T}$, then we call $\mathcal{C}$ a family of tri-circular projections ($TP$, for short). In this case, each $P_i$, $i = 0,1,2$ is called a $TP$.

The study of $GTPs$ was initiated by Abu Baker and Dutta \cite{AS} in which they defined the notion of generalized $n$-circular projections as well, and described $GTPs$ on $C(\Omega)$, the space of all continuous functions on a compact Hausdorff space $\Omega$. It was shown that if $P_0$ is a $GTP$ on $C(\Omega)$, then $\lambda_1, \lambda_2$ are cube roots of unity. In \cite{vcuka2016generalized}, {\v{C}}uka and Ili{\v{s}}evi{\'c} determined the structure of $GTPs$ on minimal norm ideals in $B(\mathcal{H})$, different from the Hilbert-Schmidt class. Hosseini \cite{H} showed that there are no $GTPs$ on the spaces of functions of bounded variation and of absolutely continuous functions. For more results on $GTPs$ as well as generalized $n$-circular projections we refer interested readers to the papers \cite{AB, AS2016, DCE, DI, DI2017, maurya2024automorphisms} and the references therein. 

Let $\mathbb{D}$ be the open unit disk. A complex-valued function on $\mathbb{D}$ is said to be analytic at a point if it has a convergent power series representation in a neighbourhood around that point. The spaces of analytic functions are fundamental in complex analysis and have applications in various fields, including functional analysis, function theory, and operator theory. Hardy spaces, Bergman spaces, Bloch spaces, and $S^p_{\mathcal{K}}$ are some examples of the space of analytic functions which are commonly studied in complex analysis. 

In this paper, we first consider a class of functional Banach spaces \cite{CM} defined as follows. 

\begin{Def}
A Banach space of complex-valued functions on a set $X$ is called a functional Banach space on X if the vector operations are pointwise operations, $f(x) = g(x)$ for each $x \in X$ implies $f = g$, $f(x) = f(y)$ for every $f$ in the space implies $x = y$, and for each $x$ in $X$, the linear functional $f \mapsto f(x)$ is continuous.
\end{Def} 

Examples of functional Banach spaces are $\ell^p$, $C[0,1]$, Bergman, Bloch, and Hardy spaces. We note that $L^p [0,1]$ is not a functional Banach space since the map $f \mapsto f(1/2)$ is not a bounded linear functional. Following \cite{botelho2009generalized} we denote by $\mathcal{F}(\mathbb{D})$ a functional Banach space on the open unit disk $\mathbb{D}$ with the property that it contains sufficiently many polynomial functions to interpolate four points, and assume that surjective isometries of $\mathcal{F}(\mathbb{D})$ are of the form 
\begin{equation} \label{IFBS}
T (f)(z)= \beta (\varphi \prime(z))^{\frac{r}{p}} f (\varphi(z))-\beta \alpha f( \varphi(0)),
\end{equation}
where $\varphi$ is a disk automorphism on $\mathbb{D}$ with the form $\varphi (z)= \mu \frac{a-z}{1-\overline{a}z}, |a| <1, |\mu|=1=|\beta|, 0 \leq |\alpha| \leq 1, p$ a positive integer, and $r$ a nonnegative real number. 

We then consider the Banach space $S^p_{\mathcal{K}}$ which consists of all analytic functions $f: \mathbb{D} \rightarrow \mathcal{K}$ such that $f' \in H^p(\mathcal{K})$ equipped with the norm $\|f\|= |f(0)| + \| f'\|_p$, see \cite {hornor1999isometrically}. Here, $\mathcal{K}$ is a separable complex Hilbert space, and $H^p(\mathcal{K})$ denotes the Hardy space of $\mathcal{K}$-valued analytic functions on $\mathbb{D}$. The form of surjective Isometries on $S^p_{\mathcal{K}}$ was characterized by Hornor and Jamison \cite {hornor1999isometrically} and is stated in the following theorem. 

\begin{thm} \label{isoform}
Let $T$ be a linear isometry of $S^p_{\mathcal{K}}$ onto $S^p_{\mathcal{K}}, p \neq 2$. Then there exist unitary operators $U$ and $V$ on $\mathcal{\mathcal{K}}$, and a disk automorphism $\varphi$ such that 
\begin{equation}\label{eqn0}
T(f)(z) =  Vf(0) + U \int_{0}^{z} [\varphi'(\xi)]^{\frac{1}{p}}f'(\varphi(\xi)) d\xi, \ \forall f  \in S^p_{\mathcal{K}} \ \text{and} \ z \in \mathbb{D}.
\end{equation}
\end{thm}

In the next two sections, we respectively characterize $GTPs$ on $\mathcal{F}(\mathbb{D})$ and $S^p_{\mathcal{K}}$. We actually find out that if $P_0$ is a $GTP$ on $\mathcal{F}(\mathbb{D})$ then $\lambda_1, \lambda_2$ are cube roots of unity. Likewise, if $P_0$ is a $GTP$ on $S^p_{\mathcal{K}}$ then either $\lambda_1, \lambda_2$ are cube roots of unity or $P_0$ is a $TP$. 

We recall the following technical lemma, which will be used in this paper. 

\begin{lem} \label{Lem0}	
Let $X$ be a Banach space, and let $\mathcal{C} = \{P_0,P_1, P_2\}$ be a collection of nonzero distinct projections. Let $\lambda_1, \lambda_2$ be distinct complex numbers not equal to $1$. Then the following conditions are equivalent.
	\begin{enumerate}
		\item $T = P_0 + \lambda_1 P_1 + \lambda_2 P_2$ and $\mathcal{C}$ defines an orthogonal partition of identity by projections.
		\item The following holds: 
		$ (T- I)(T -\lambda_1 I)(T - \lambda_2 I) =0 $ and
		$$ P_0 = \frac{(T - \lambda_1 I)(T- \lambda_2 I)}{(1 - \lambda_1)(1 - \lambda_2)}, P_1 = \frac{(T - I)(T- \lambda_2 I)}{(\lambda_1 - 1)(\lambda_1 - \lambda_2)}, P_2 = \frac{(T - I)(T- \lambda_1 I)}{(\lambda_2 - 1)(\lambda_2 - \lambda_1)}. $$
	\end{enumerate}
\end{lem}
	
\section{Generalized tri-circular projections on $\mathcal{F}(\mathbb{D})$}

In this section, we state and prove our first result which characterizes $GTP$ on $\mathcal{F}(\mathbb{D})$.
	
\begin{thm} \label{main3}
Let $\mathcal{C} = \{P_0, P_1, P_2\}$ be a family of generalized tri-circular projections on $\mathcal{F}(\mathbb{D})$. Then one of the following holds: 

\begin{enumerate}
\item $\varphi(0)=0$. Then either
\begin{enumerate}
\item $\mu = 1$ and one of the members from $\mathcal{C}$ becomes a generalized bi-circular projection, or

\item $\mu^3 = - 1$, $\mu \neq - 1$, $\beta^3 = 1$ and $\lambda_1$, $\lambda_2$ are cube roots of unity. Moreover, $\alpha=0$ or $\alpha$ is a root of the equation $\alpha^2-3 \alpha (-\mu)^{\frac{r}{p}} + 3 (\mu^2)^{\frac{r}{p}}=0$.
\end{enumerate}

\item $\varphi(0) \neq 0, \varphi^2(0)= 0$. In this case, one of the members from $\mathcal{C}$ becomes a generalized bi-circular projection.

\item $\varphi(0) \neq 0, \varphi^2(0)\neq 0$ and $\varphi^3(z) = z$. Then $\lambda_1$, $\lambda_2$ are cube roots of unity, $\alpha = 0$, and $ \beta^3 = 1$. 
\end{enumerate}
\end{thm}

\begin{proof}
Let $\mathcal{C} = \{P_0, P_1, P_2\}$ be a family of generalized tri-circular projections on $\mathcal{F}(\mathbb{D})$. Then there exist distinct scalars $\lambda_1, \lambda_2 \in \mathbb{T}\setminus \{1\}$ such that $P_0 +\lambda_1 P_1 + \lambda_2 P_2 = T$. Thus,  $(T-I)(T -\lambda_1 I)(T - \lambda_2 I) =0$ or 
\begin{equation} \label{equ2}
T^3-(1+l) T^2 + (l+m) T -m I=0, \text{ where } l= \lambda_1 + \lambda_2, \; m = \lambda_1 \lambda_2.
\end{equation} 
Now,
\begin{align*}
T f(z) &= \beta (\varphi ^{\prime} (z))^{\frac{r}{p}} f(\varphi(z))- \beta \alpha f(\varphi(0)), \numberthis \label{equ3} \\
(T^2 f)(z)&= \beta ^2 \big[[(\varphi^2) ^ {\prime}(z)] ^{\frac{r}{p}} f(\varphi^2(z))- \alpha [\varphi^{\prime}(z)]^{\frac{r}{p}} f(\varphi(0)) - \alpha [\varphi^{\prime}(\varphi(0))]^{\frac{r}{p}} f(\varphi^2(0)) \big], \numberthis \label{equ4}
\end{align*}
and      
\begin{align*}
 (T^3 f)(z)) &= \beta ^3  [(\varphi^3) ^ {\prime}(z) ]^{\frac{r}{p}} f(\varphi^3(z))- \alpha \beta ^3  [(\varphi^2) ^ {\prime}(z)] ^{\frac{r}{p}} f(\varphi(0)) - \alpha \beta^3 [\varphi^{\prime}(z) \varphi ^{\prime} (\varphi(0)) ]^{\frac{r}{p}}\\ \noindent & f(\varphi^2(0)) + \alpha^2 \beta^3 [\varphi^{\prime}(z)]^{\frac{r}{p}} f(\varphi(0))- \alpha \beta^3 [\varphi^{\prime}(\varphi(0)) \varphi ^{\prime} (\varphi^2(0)) ]^{\frac{r}{p}} f(\varphi^3(0)) \\ \noindent &+ \alpha^2 \beta^3 [\varphi^{\prime}(\varphi(0))]^{\frac{r}{p}} f(\varphi(0)) + \alpha^2 \beta ^3 [\varphi ^{\prime} (\varphi(0))]^{\frac{r}{p}} f(\varphi^2(0)) - \alpha^3 \beta^3 f(\varphi(0)). \numberthis \label{equ5} 
\end{align*}
Now, Equation \eqref{equ2} becomes
\begin{align}
\beta ^3  [(\varphi^3) ^ {\prime}(z) ] ^{\frac{r}{p}} f(\varphi^3(z)) - \alpha \beta ^3  [(\varphi^2) ^ {\prime}(z)] ^{\frac{r}{p}} f(\varphi(0)) - \alpha \beta^3 [\varphi^{\prime}(z) \varphi ^{\prime} (\varphi(0)) ]^{\frac{r}{p}} f(\varphi^2(0)) & \nonumber \\ 
 + \alpha^2 \beta^3 [\varphi^{\prime}(z)]^{\frac{r}{p}} f(\varphi(0))- \alpha \beta^3 [\varphi^{\prime}(\varphi(0)) \varphi ^{\prime} (\varphi^2(0)) ]^{\frac{r}{p}} f(\varphi^3(0))+ \alpha^2 \beta^3 [\varphi^{\prime}(\varphi(0))]^{\frac{r}{p}} f(\varphi(0)) &\nonumber \\  + \alpha^2 \beta ^3 [\varphi ^{\prime} (\varphi(0))]^{\frac{r}{p}} f(\varphi^2(0)) - \alpha^3 \beta^3 f(\varphi(0)) - (1+l) \beta ^2 \big[ [(\varphi^2) ^ {\prime}(z)] ^{\frac{r}{p}} f(\varphi^2(z)) & \nonumber \\ - \alpha [\varphi^{\prime}(z)]^{\frac{r}{p}} f(\varphi(0))  - \alpha [\varphi^{\prime}(\varphi(0))]^{\frac{r}{p}} f(\varphi^2(0)) + \alpha^2  f(\varphi(0))\big] + & \nonumber \\ (l+m) \beta \big[ [\varphi ^{\prime} (z)]^{\frac{r}{p}} f(\varphi(z))- \alpha f(\varphi(0)) \big]- m f(z) =0. \label{equ6} 
\end{align}
Consider the following three cases.

\begin{enumerate}
\item $\varphi(0)=0$
\item $\varphi(0)\neq 0, \varphi^2(0)=0$
\item $\varphi(0) \neq 0, \varphi^2(0)\neq 0$ and $\varphi^3(z) = z$ for every $z\in \mathbb{D}$.
\end{enumerate}
			
\subsection*{Case (1)} $\varphi(0)=0$. This implies that $\varphi (z)= -\mu z$ for all $z \in \mathbb{D}$. Assume that $\mu \neq -1$, i.e., $\varphi(z) \neq z$ (we will deal with the sub-case $\varphi(z)=z$ separately). Using $\varphi (z)= -\mu z$ in Equation \eqref{equ6} we obtain,
\begin{align*}
\beta ^3  [-\mu ^3] ^{\frac{r}{p}} f(-\mu^3 z)- (1+l) \beta ^2 [\mu^2] ^{\frac{r}{p}} f(\mu ^2 z)+(l+m) \beta [-\mu]^{\frac{r}{p}} f(-\mu z )+  \big[3 \alpha^2 \beta^3 [-\mu] ^{\frac{r}{p}} & \noindent \\ -3\alpha \beta ^3  [\mu^2]  ^{\frac{r}{p}}-\alpha^3 \beta^3 + 2(1+l) \alpha  \beta ^2 [-\mu] ^{\frac{r}{p}}- (1+l) \alpha^2  \beta ^2 - (l+m) \alpha \beta \big]f(0)-m f(z)=0 & \noindent \numberthis \label{eq2a}
\end{align*}

Let $z_0$ be a nonzero element in $\mathbb{D}$. Now we have three possible sub-cases:
\begin{enumerate}[label=(\alph*)]
	\item $\mu=1$ 
	\item $\mu= -1$ 
	\item $\mu \neq \pm 1$
\end{enumerate}

\subsection*{Sub-case (a)} $\mu=1$. So, Equation \eqref{eq2a} becomes  
\begin{align*}
(- 1)^{\frac{r}{p}} \beta [\beta ^2 + l+m]f(-z)-  [(1+l) \beta ^2 + m]f(z)+  \big[3 \alpha^2 \beta^3 (-1)^{\frac{r}{p}} & \noindent \\ -3\alpha \beta ^3  -\alpha^3 \beta^3 + 2(1+l) \alpha  \beta ^2 (- 1)^{\frac{r}{p}}- (1+l) \alpha^2  \beta ^2 - (l+m) \alpha \beta \big]f(0)=0 & \noindent \numberthis \label{eq2a'}
\end{align*}
We can choose appropriate polynomials to get $\beta ^2 + l+m = 0$ and $(1+l) \beta ^2 + m = 0$. Solving these two equations we get $(\lambda_1+\lambda_2)(\lambda_1+1) (\lambda_2+1)=0$. Thus, $\lambda_1= - \lambda_2$, $\lambda_2=-1$ or $\lambda_1=-1$, which implies that $P_0, P_1$,  or $P_2$ respectively, becomes a generalized bi-circular projection.

\subsection*{Sub-case (b)} $\mu = -1$, i.e., $\varphi(z) =z$. Equation \eqref{equ6} takes the form
\begin{align} \label{eq5a}
[\beta ^3- (1+l) \beta ^2 +(l+m) \beta-m] f(z) + \big[\beta^3 (3\alpha^2 - 3 \alpha  - \alpha^3) + (1+l) \beta^2 ( 2 \alpha- \alpha ^2)  &\nonumber \\ -  \alpha \beta (l+m) \big]f(0) =0.
\end{align}
Then 
\begin{equation}
\label{eq6}
\beta ^3- (1+l) \beta ^2 +(l+m) \beta-m=0. 
\end{equation}
It follows that $\beta = 1, \lambda_1$ or $\lambda_2$. Using these values in Equation \eqref{eq5a} we conclude the following.

If $\beta=1$, then $\alpha f(0)=0$, $\alpha= 1-\lambda_1$ or $\alpha= 1-\lambda_2$. This implies that $P_3=0$. 

If $\beta= \lambda_1$, then  $\alpha f(0)=0$, $\alpha= 1- \compconj{\lambda_1}$ or $\alpha= 1- \compconj{\lambda_1} \lambda_2$. This implies that $P_1=0$ or $P_3=0$. 

If $\beta= \lambda_2$, then  $\alpha f(0)=0$, $\alpha= 1- \compconj{\lambda_2}$ or $\alpha= 1- \compconj{\lambda_2} \lambda_1$. This implies that $P_1=0$ or $P_2=0$.

Since we considered only nonzero projections, this sub-case is not possible, that is, $\mu \neq -1$. 
		
\subsection*{Sub-case (c)} $\mu \neq \pm 1$. Then $\mu^2 z_0 \notin \{ 0, z_0, -\mu^3 z_0, -\mu z_0\}$. We choose polynomials $f$ and $g$ in Equation \eqref{eq2a} such that $$f(0)=f(-\mu^3 z_0)= f(-\mu z_0)= f(z_0)=0 \text{ and } f(\mu ^2 z_0)=1;$$ 
$$g(0)=g(\mu^2 z_0)= g(-\mu^3 z_0)= g(z_0)=0 \text{ and } g(-\mu z_0) =1,$$
to get $ l= -1$ and $l+m = 0$, respectively. This implies that $\lambda_1$ and $\lambda_2$ are cube roots of unity. It also follows from Equation \eqref{eq2a} that $- \mu ^3 = 1$, and hence $\beta^3=1$. In this case, either $\alpha=0$ or $\alpha$ is a root of the equation $\alpha^2-3 \alpha (-\mu)^{\frac{r}{p}} + 3 (\mu^2)^{\frac{r}{p}}=0$.

\subsection*{Case (2)} $\varphi(0) \neq 0$, but $\varphi^2(0)=0$. This implies $\varphi^2(z)= z$ and $\mu=1$. Now we can rewrite the Equation \eqref{equ6} as  
\begin{align*}
\beta ^3  [\varphi ^ {\prime}(z)] ^{\frac{r}{p}} f(\varphi(z))-\alpha \beta^3 f(\varphi(0)) - \alpha \beta ^3  [\varphi ^ {\prime}(z) \varphi^{\prime}(\varphi(0))] ^{\frac{r}{p}} f(0)  + \alpha^2 \beta^3 [\varphi^{\prime}(z)]^{\frac{r}{p}} f(\varphi(0)) &\noindent \\ - \alpha \beta^3 [\varphi ^ {\prime}(0) \varphi^{\prime}(\varphi(0)) ]^{\frac{r}{p}} f(\varphi(0))  + \alpha^2 \beta^3 [\varphi^{\prime}(\varphi(0))]^{\frac{r}{p}} f(\varphi(0)) + \alpha^2 \beta ^3 {\varphi ^{\prime} (\varphi(0))}^{\frac{r}{p}} f(0) - \alpha^3 \beta^3 f(\varphi(0)) &\noindent \\ - (1+l) \beta ^2 \left( f(z)- \alpha [(\varphi) ^ {\prime}(z)] ^{\frac{r}{p}} f(\varphi(0))- \alpha [\varphi^{\prime}(\varphi(0))]^{\frac{r}{p}} f(0)  + \alpha^2 f(\varphi(0))\right) &\noindent \\ +(l+m) \beta\left((\varphi ^{\prime} (z))^{\frac{r}{p}} f(\varphi(z))- \alpha f(\varphi(0))\right)- m f(z) =0. \numberthis \label{Eq1}
\end{align*}		
		
For a nonzero $z_0$ different from $\varphi(0)$, we choose a polynomial $f$ such that $f(z_0)=1, f(0)=f(\varphi(0))=f(\varphi(z_0))= 0$ to get $\beta^2 (1+l) = -m $. Again, for a polynomial $g$ such that $ g(\varphi(z_0))=1, g(0)=g(\varphi(0))=g(z_0)= 0$ we obtain $\beta^2 = -(l+m)$. It follows that $(\lambda_1+\lambda_2)(\lambda_1+1)(\lambda_2+1)=0$, which implies that $\lambda_1= - \lambda_2$ or $\lambda_2=-1$ or $\lambda_1=-1$. Therefore, $P_0$ or $P_1$ or $P_2$, respectively, becomes a generalized bi-circular projection.
		
\subsection*{Case (3)} $\varphi(0) \neq 0, \varphi^2(0)\neq 0$ and $\varphi^3(z) = z$. In this case,  Equation (\ref{equ6}) takes the form
\begin{align}
\left[\beta ^3  [(\varphi^3) ^ {\prime}(z)] ^{\frac{r}{p}} - m \right] f(z)- \alpha \beta ^3  ([(\varphi^2) ^ {\prime}(z)] ^{\frac{r}{p}} f(\varphi(0)) - \alpha \beta^3 [\varphi^{\prime}(z) \varphi ^{\prime} (\varphi(0)) ]^{\frac{r}{p}} f(\varphi^2(0)) & \nonumber \\ 
+ \alpha^2 \beta^3 [\varphi^{\prime}(z)]^{\frac{r}{p}} f(\varphi(0))- \alpha \beta^3 [\varphi^{\prime}(\varphi(0)) \varphi ^{\prime} (\varphi^2(0))]^{\frac{r}{p}} f(0)+ \alpha^2 \beta^3 [\varphi^{\prime}(\varphi(0))]^{\frac{r}{p}} f(\varphi(0)) &\nonumber \\  
+ \alpha^2 \beta ^3 (\varphi ^{\prime} (\varphi(0)))^{\frac{r}{p}} f(\varphi^2(0)) - \alpha^3 \beta^3 f(\varphi(0)) - (1+l) \beta ^2 [[(\varphi^2) ^ {\prime}(z)] ^{\frac{r}{p}} f(\varphi^2(z)) & \nonumber \\  
- \alpha [\varphi^{\prime}(z)]^{\frac{r}{p}} f(\varphi(0)) - \alpha [\varphi^{\prime}(\varphi(0))]^{\frac{r}{p}} f(\varphi^2(0)) + \alpha^2  f(\varphi(0))]   +(l+m) \beta [(\varphi ^{\prime} (z))^{\frac{r}{p}} f(\varphi(z)) & \nonumber \\ 
- \alpha f(\varphi(0))] =0. \numberthis \label{equ7}
\end{align}
		
For a nonzero $z_0 \in \mathbb{D}$, consider the set $ \Gamma= \{ 0, z_0, \varphi(0), \varphi^2(0), \varphi(z_0), \varphi^2(z_0) \}$. Choose polynomials $f, g$ and $h$ in $\mathcal{F}(\mathbb{D})$ such that $f(z_0) = 1$, $g(\varphi(z_0)) = 1$, $h(\varphi^2(z_0)) = 1$, and are respectively zero at all other points of the set $\Gamma$. Using the functions $f, g$ and $h$, Equation \eqref{equ7} respectively reduces to $\beta ^3  = m$, $(l + m)\beta [\varphi^{'}(z)]^\frac{r}{p} = 0$ and $	-(1 + l)\beta^2 [ (\varphi^2)^\prime (z))]^\frac{r}{p} = 0$. Since $\varphi^{'}(z)\neq 0$ for all $z \in \mathbb{D}$, we conclude that $1 + \lambda_1 + \lambda_2 = 0$ and $\lambda_1 + \lambda_2 + \lambda_1\lambda_2 = 0$. This implies that $\lambda_1$ and $\lambda_2$ are cube roots of unity. Using all these information in Equation \eqref{equ7} and observing that $\varphi^{\prime}(\varphi(0)) \varphi ^{\prime} (\varphi^2(0)) \neq 0$, we obtain $\alpha = 0$. 

This completes the proof of the theorem. 
 \end{proof}

\section{Generalized tri-circular projections on $S^p_{\mathcal{K}}$}

In this section, we state and prove our second result which characterizes $GTP$ on $S^p_{\mathcal{K}}$.

\begin{thm} \label{main4}
Let $\mathcal{C} = \{P_0, P_1, P_2\}$ be a family of generalized tri-circular projections on $S^p_{\mathcal{K}}, p \neq 2$. Then one of the following holds: 
\begin{enumerate}
\item $\varphi (z) = z$. Then $\mathcal{C}$ is a family of tri-circular projections. Moreover, the point spectrum of $U$ and $V$ is $\sigma_p = \{1, \lambda_1, \lambda_2\}$.

\item $\varphi(z) \neq z, \varphi^2(z) =z$. In this case, one of the members from $\mathcal{C}$ becomes a generalized bi-circular projection.

\item $\varphi(z) \neq z, \varphi^2(z) \neq z, \varphi^3(z) =z$. Then $\lambda_1$ and $\lambda_2$ are cube roots of unity, and $U^3 = V^3 = I_{\mathcal{K}}$, where $I_{\mathcal{K}}$ denotes the identity operator on $\mathcal{K}$.
\end{enumerate}
\end{thm}			
\begin{proof}
Let $\mathcal{C} = \{P_0, P_1, P_2\}$ be a family of generalized tri-circular projections on $S^p_{\mathcal{K}}$. Then there exist distinct scalars $\lambda_1, \lambda_2 \in \mathbb{T}\setminus \{1\}$ such that $P_0 +\lambda_1 P_1 + \lambda_2 P_2 = T$.	
	
By Theorem \ref{isoform}, we have
\begin{equation*}
T(f)(z) =  Vf(0) + U \int_{0}^{z} [\varphi'(\xi)]^{\frac{1}{p}}f'(\varphi(\xi)) d\xi, \ \forall f  \in S^p_{\mathcal{K}} \ \text{and} \ z \in \mathbb{D}.
\end{equation*}
Simple computations show that
\begin{align*}
(Tf)(z) =&  Vf(0) + U \int_{0}^{z} [\varphi'(\xi)]^{\frac{1}{p}}f'(\varphi(\xi)) d\xi, \\
(T^2f)(z) =&  V^2f(0) + U^2 \int_{0}^{z} [(\varphi^2)'(\xi)]^{\frac{1}{p}}f'(\varphi^2(\xi)) d\xi, \\
(T^3f)(z) =&  V^3f(0) + U^3 \int_{0}^{z} [(\varphi^3)'(\xi)]^{\frac{1}{p}}f'(\varphi^3(\xi)) d\xi. 
\end{align*}
Lemma \ref{Lem0} gives the following 
\begin{equation*}
T^3f(z)-(1+l) T^2f(z) + (l+m) Tf(z) -mf(z) = 0,\  \forall f \in S^p_{\mathcal{K}} \ \text{and} \ z \in \mathbb{D} 
\end{equation*}
where $ l= \lambda_1 + \lambda_2, \; m = \lambda_1 \lambda_2$.

Putting the values of $T, T^2$ and $T^3$,
\begin{align} \label{mi}
\Big[V^3f(0) + U^3 \int_{0}^{z} [(\varphi^3)'(\xi)]^{\frac{1}{p}} f'(\varphi^3(\xi)) d\xi\Big]  - (1+l)\Big[ V^2f(0) + U^2 \int_{0}^{z} [(\varphi^2)'(\xi)]^{\frac{1}{p}} f'(\varphi^2(\xi)) d\xi \Big]  \notag \\
+(l+m) \Big[ Vf(0) + U \int_{0}^{z}[\varphi'(\xi)] ^{\frac{1}{p}} f'(\varphi(\xi)) d\xi \Big] - m f(z) =0.
\end{align}
After differentiation, we get
\begin{align} \label{me2}
U^3 [(\varphi^3)'(z)]^{\frac{1}{p}} f'(\varphi^3(z)) -(1+l) U^2[(\varphi^2)'(z)]^{\frac{1}{p}} f'(\varphi^2(z)) & \nonumber \\ 
+ (l+m) U[\varphi'(z)] ^{\frac{1}{p}} f'(\varphi(z)) -mf'(z) =0.
\end{align}
We examine the following three possible cases:
 
\begin{enumerate}
\item $\varphi(z) = z$,
\item $\varphi(z) \neq z, \varphi^2(z) =z$,
\item $\varphi(z) \neq z, \varphi^2(z) \neq z, \varphi^3(z) =z$.          
\end{enumerate}
\subsection*{Case (1)} $\varphi(z)= z$. 

Equation \eqref{me2} reduces 
\begin{align*}
V^3f(0) + U^3 [f(z)-f(0)] - (1+l)\Big[ V^2f(0) + U^2 [f(z)-f(0)] \Big]  \\
+(l+m) \Big[ Vf(0) + U [f(z)-f(0)] \Big] - m f(z) =0.
\end{align*}
On simplifying,
\begin{align*}
[U^3 -(1+l) U^2 + (l+m)U -m I_{\mathcal{K}}] f(z) + [V^3 -(1+l) V^2   \\ 
+ (l+m)V -[U^3 -(1+l) U^2 + (l+m)U] f(0) &= 0.
\end{align*}
By choosing suitable functions, we get the following equations.
\begin{equation*}
U^3 -(1+l) U^2 + (l+m)U -m I_{\mathcal{K}} =0 \text{ or } (U - I_{\mathcal{K}})(U - \lambda_1 I_{\mathcal{K}})(U - \lambda_2 I_{\mathcal{K}}) = 0,
\end{equation*}
and 
\begin{equation*}
V^3 -(1+l) V^2 + (l+m)V -m I_{\mathcal{K}} =0 \text{ or } (V - I_{\mathcal{K}})(V - \lambda_1 I_{\mathcal{K}})(V - \lambda_2 I_{\mathcal{K}}) = 0.
\end{equation*}
Therefore, the point spectrum of $U$ and $V$ is $\sigma_p = \{1, \lambda_1, \lambda_2\}$. Moreover, $\mathcal{C}$ is a family of tri-circular projections.
\subsection*{Case (2)} $\varphi(z) \neq z, \varphi^2(z) =z$.
			
From Equation \eqref{me2}, we have
\begin{equation} \label{mi2}
U^3 [(\varphi)'(z)]^{\frac{1}{p}} f'(\varphi(z)) -(1+l) U^2 f'(z) + (l+m) U[\varphi'(z)] ^{\frac{1}{p}} f'(\varphi(z)) -mf'(z) =0.
\end{equation}

Again, for a fixed nonzero vector $v \in \mathcal{K}$, we choose the functions $f'_1(\xi)= v$ and  $f'_2(\xi) = \xi v$  for all $\xi \in \mathbb{D}$. Evaluating Equation \ref{mi2} for $f'_1$ and $f'_2$, we get the following equations, respectively
\begin{equation} \label{f4}
[(\varphi)'(z)]^{\frac{1}{p}} 	U^3v -(1+l) U^2 v + (l+m) [\varphi'(z)] ^{\frac{1}{p}} Uv -mv =0,
\end{equation}
\begin{equation} \label{f5}
[(\varphi)'(z)]^{\frac{1}{p}} \varphi(z) 	U^3v -(1+l)z U^2v + (l+m) [\varphi'(z)] ^{\frac{1}{p}} \varphi(z) Uv -m zv =0.
\end{equation}
			
Multiplying Equation \eqref{f4} by $\varphi(z)$ and then subtracting it from \eqref{f5}, we arrive at
\begin{equation*} 
(z - \varphi(z)) [ -(1+l) U^2v -mv]=0.
\end{equation*}
Since $\varphi(z) \neq z$, therefore  
\begin{equation} \label{f6}
mv = -(1+l) U^2v.
\end{equation}			
Again by the Equation \eqref{f4}, we have 			
\begin{equation*} 
(1+l) U^2 v = [(\varphi)'(z)]^{\frac{1}{p}} 	U^3v  + (l+m) [\varphi'(z)] ^{\frac{1}{p}} Uv -mv.
\end{equation*}			
Putting this in the Equation \eqref{f5}, we get
\begin{equation*} 
(\varphi(z)-z) [ (\varphi)'(z)]^{\frac{1}{p}}	[U^3v  + (l+m) Uv]  =0.
\end{equation*}			
Here $\varphi(z) \neq z$, thus			
\begin{equation} \label{f7}
U^3v  =- (l+m)Uv.
\end{equation}		
Simplifying the Equations \ref{f6} and \ref{f7}, we obtain $(1+l)(l+m) -m =0$, or $(\lambda_1+1)(\lambda_2+1)(\lambda_1+ \lambda_2) =0$. This implies that $\lambda_1= - \lambda_2$ or $\lambda_2=-1$ or $\lambda_1=-1$. Therefore, $P_0$ or $P_1$ or $P_2$, respectively, becomes a generalized bi-circular projection.

\subsection*{Case (3)} $\varphi(z) \neq z, \varphi^2(z) \neq z, \varphi^3(z) =z$. 
		
Equation \eqref{me2} becomes			
\begin{equation*} 
U^3 f'(z) -(1+l) U^2[(\varphi^2)'(z)]^{\frac{1}{p}} f'(\varphi^2(z)) +(l+m) U[\varphi'(z)] ^{\frac{1}{p}} f'(\varphi(z)) -mf'(z) =0.
\end{equation*}

For a fixed nonzero vector $v \in \mathcal{K}$, select the functions $f'_1(\xi) = v, f'_2(\xi)= \xi v $ and $f'_3(\xi) = \xi^2v$ for all $\xi \in \mathbb{D}$. We obtain the following equations, respectively.         
\begin{equation} \label{f1}
U^3v -(1+l) [(\varphi^2)'(z)]^{\frac{1}{p}} U^2v +(l+m)[\varphi'(z)] ^{\frac{1}{p}} Uv -mv =0.
\end{equation}	
\begin{equation}\label{f2}
z U^3v  -(1+l) [(\varphi^2)'(z)]^{\frac{1}{p}} \varphi^2(z)U^2v +(l+m) [\varphi'(z)] ^{\frac{1}{p}} \varphi(z)Uv -mzv =0.
\end{equation}			
\begin{equation} \label{f3}
z^2 U^3v  -(1+l) [(\varphi^2)'(z)]^{\frac{1}{p}} (\varphi^2(z))^2U^2v +(l+m) [\varphi'(z)] ^{\frac{1}{p}} (\varphi(z))^2Uv -mz^2v =0.
\end{equation}

Multiplying Equations \eqref{f1} and \eqref{f2} by $z$, then subtracting the Equation \eqref{f3} from \eqref{f2} and Equation \eqref{f2} from \eqref{f1} we get 
\begin{equation}\label{c11}
(1+l) [(\varphi^2)'(z)]^{\frac{1}{p}}[z - \varphi^2(z)] = (l+m)[\varphi'(z)] ^{\frac{1}{p}} [z- \varphi(z)] 
\end{equation}
and
\begin{equation} \label{c2}
(1+l) [(\varphi^2)'(z)]^{\frac{1}{p}} \varphi^2(z)[z - \varphi^2(z)] = (l+m)[\varphi'(z)] ^{\frac{1}{p}} \varphi(z)[z- \varphi(z)] .
\end{equation}
From above Equations \eqref{c11} and \eqref{c2}, we obtain
\begin{equation*}
(l+m)[\varphi'(z)] ^{\frac{1}{p}} (z-\varphi(z))(\varphi(z) - \varphi^2(z)) =0.
\end{equation*}
Since $[\varphi'(z)]^{\frac{1}{p}} \neq 0$ . It follows that			
\begin{equation*}
(l+m) (z-\varphi(z))(\varphi(z) - \varphi^2(z)) =0.
\end{equation*}
Since, $z \neq \varphi(z)$, we have $l+m =0$, i.e. $\lambda_1 + \lambda_2 + \lambda_1 \lambda_2 =0$. Similarly, $1+l =0$, i.e., $1 +	\lambda_1 + \lambda_2 =0$. This implies that $\lambda_1$ and $\lambda_2$ are cube roots of unity. Moreover, $T^3=I$. This further implies that 
\begin{equation*}
V^3f(0) + U^3 \int_{0}^{z} f'(\xi) d\xi  =f(z),
\end{equation*}
or
\begin{equation} \label{mi1}
(U^3 - I_{\mathcal{K}})f(z) + (V^3-U^3) f(0) = 0,
\end{equation} 
where $I_{\mathcal{K}}$ is an identity operator on $\mathcal{K}$. Therefore, $U^3 = V^3 = I_{\mathcal{K}}$. 

This completes the proof of the theorem.
\end{proof}

\end{document}